\documentclass[a4paper, 12pt, reqno, dvipsnames]{amsart}

\usepackage[normalem]{ulem} 
\usepackage{lmodern}
\usepackage[english]{babel}
\usepackage[latin1]{inputenc}
\usepackage{microtype}
\usepackage{amsfonts,amssymb,amsmath,amsthm,mathrsfs}
\usepackage{mathtools}
\usepackage{color, graphicx}
\usepackage[all]{xy}
\usepackage[bookmarks=false]{hyperref}  
\usepackage[inline]{enumitem}
\usepackage{cleveref}
\usepackage{a4wide}
\usepackage{upgreek}
\usepackage{etoolbox}

\makeatletter
\patchcmd{\@setaddresses}{\indent}{\noindent}{}{}
\patchcmd{\@setaddresses}{\indent}{\noindent}{}{}
\patchcmd{\@setaddresses}{\indent}{\noindent}{}{}
\patchcmd{\@setaddresses}{\indent}{\noindent}{}{}
\makeatother

\makeatletter
\DeclareMathSizes{\@xpt}{\@xpt}{6}{5}
\makeatother

\makeatletter
\def\namedlabel#1#2{\begingroup
    #2%
    \def\@currentlabel{#2}%
    \phantomsection\label{#1}\endgroup
}
\makeatother

\allowdisplaybreaks

\theoremstyle{plain}

\newtheorem{theorem}{Theorem}[section]
\newtheorem{lemma}[theorem]{Lemma}
\newtheorem{proposition}[theorem]{Proposition}
\newtheorem{corollary}[theorem]{Corollary}
\newtheorem*{theorem*}{Theorem}

\theoremstyle{definition}
\newtheorem{definition}[theorem]{Definition}
\newtheorem{example}[theorem]{Example}

\theoremstyle{remark}
\newtheorem{remark}[theorem]{Remark}

\newcommand{\R}{\mathbb{R}}

\newcommand{\K}{\Bbbk}

\newcommand{\op}[1]{#1^{\mathrm{op}}} 
\newcommand{\mf}[1]{\mathfrak{#1}} 
\newcommand{\ms}[1]{\mathsf{#1}}

\renewcommand{\ker}{\mathsf{ker}} 

\newcommand{\id}{\mathsf{Id}} 

\newcommand{\alg}{\mathsf{Alg}}
\newcommand{\algk}{\alg_{\K}} 
\newcommand{\Set}{{\mathsf{Set}}} 
\newcommand{\Top}{{\mathsf{Top}}} 
\newcommand{\Haus}{{\mathsf{Haus}}} 
\newcommand{\Rmod}[1]{\M_{#1}} 
\newcommand{\gpmod}[1]{\mathsf{gPMod}_{#1}} 
\renewcommand{\mod}[1]{\mathsf{Mod}_{#1}} 

\newcommand{\cC}{{\mathcal C}}

\newcommand{\cI}{{\mathcal I}}

\newcommand{\M}{\mathsf{Mod}} 

\newcommand{\ie}{i.e.~}
\newcommand{\eg}{e.g.~}

\renewcommand{\top}{\mathbf{\uptau}}

\newcommand{\ot}{\otimes}
\newcommand{\bul}{\bullet}

\newbox\pullbackbox
\setbox\pullbackbox=\hbox{\xy 0;<1mm,0mm>: \POS(2,0)\ar@{-} (0,2) \ar@{-} (4,2)
\endxy}

\newbox\pushoutbox
\setbox\pushoutbox=\hbox{\xy 0;<1mm,0mm>: \POS(2,2)\ar@{-} (0,0) \ar@{-} (4,0)
\endxy}
\def\pushout{\copy\pushoutbox}

\newdir{ >}{{}*!/-10pt/@{>}}

\definecolor{bostonuniversityred}{rgb}{0.8, 0.0, 0.0}

\title[Globalization of geometric partial (co)modules in TOP and ALG]{On the globalization of geometric partial (co)modules in the categories of topological spaces and algebras}

\author{Paolo Saracco}
\address{D\'epartement de Math\'ematique, Universit\'e Libre de Bruxelles, Boulevard du Triomphe, B-1050 Brussels, Belgium.}
\urladdr{\url{sites.google.com/view/paolo-saracco}}
\urladdr{\url{homepages.ulb.ac.be/~psaracco}}
\email{paolo.saracco@ulb.be}

\author{Joost Vercruysse}
\address{D\'epartement de Math\'ematique, Universit\'e Libre de Bruxelles, Boulevard du Triomphe, B-1050 Brussels, Belgium.}
\urladdr{\url{homepages.ulb.ac.be/~jvercruy}}
\email{jvercruy@ulb.be}

\thanks{PS is a Charg\'e de Recherches of the Fonds de la Recherche Scientifique - FNRS and a member of the National Group for Algebraic and Geometric Structures and their Applications (GNSAGA-INdAM).
JV thanks the FNRS (National Research Fund of the French speaking community in Belgium) for support via the MIS project `Antipode' (Grant F.4502.18).}

\begin{document}

\begin{abstract}
We study the globalization of partial actions on sets and topological spaces and of partial coactions on algebras by applying the general theory of globalization for geometric partial comodules, as previously developed by the authors. We show that this approach does not only allow to recover all known results in these settings, but it allows to treat new cases of interest, too. 
\end{abstract}

\maketitle



\section{Introduction}

Since the very beginning of the theory of partial group actions \cite{Exel1}, one of the main questions has been to understand if any given partial action can be obtained as a restriction of a classical (global) group action \cite{Abadie,AlvesBatista,Hollings, KellendorkLawason, Megrelishvili}.

The {\em geometric partial (co)modules} from \cite{JoostJiawey} provide a general categorical framework to study all sorts of partial actions in a unified way, subsuming partial actions of groups as well as partial (co)representations of Hopf algebras (see  \cite{Saracco-Vercruysse3} for a detailed treatment of the globalization question in these cases). Moreover, geometric partial comodules also allow to treat cases that cannot be described by the Hopf-algebraic partial (co)actions from
 \cite{Caenepeel-Janssen}, such as genuine partial actions of algebraic groups on irreducible varieties. In a previous paper \cite{Saracco-Vercruysse}, we defined and studied globalizations for geometric partial comodules. In the present paper we apply the general results from \cite{Saracco-Vercruysse} to discuss in more detail the globalization results for partial actions of topological monoids on topological spaces and partial comodule algebras. We show that our approach (Theorem \ref{thm:topmain}) not only allows to recover and unify the globalization  results for topological partial actions from \cite{Abadie} and \cite{Megrelishvili} (see Corollary \ref{topactisglob}), but it also allows to treat globalizations in new cases (see Example \ref{ggpmnotparact}). 
Next, we consider the partial comodule algebras (also called partial coactions) over bialgebras from \cite{Caenepeel-Janssen} and show that these are globalizable (Theorem \ref{prop:globPCA}). Finally we explain how our globalization for partial comodule algebras differs from the enveloping coaction from \cite{AlvesBatista} (Proposition \ref{pr:comodalg}).

We denote identity on an object $X$ by $\id_X$ or simply by $X$ itself.

\section{The motivating example: Partial actions of monoids on sets}

\subsection{Categorical formulation of partial actions}
\label{se:partact}

A partial action of a monoid on a set is, intuitively, a ``partially defined'' action, satisfying unitality and associativity conditions, whenever these make sense. Let us make this more explicit.
Fix a monoid $M$ with composition law $\Delta:M\times M\to M$ and neutral element $u:\{*\} \to M, *\mapsto e$. 
A {\em partial action datum} over $M$ is a quadruple $(X,X\bul M,\pi_X,\rho_X)$ consisting of two sets, $X$ and $X \bul M$, and of a span  
\begin{equation}\label{datum}
\begin{gathered}
\xymatrix @C=18pt @R=8pt {
X\times M & & X \\
 & X\bul M \ar@{ >->}[ul]^-{\pi_X} \ar[ur]_-{\rho_X} & 
}
\end{gathered}
\end{equation}
in $\Set$, where $\pi_X$ is an injective map. The set $X\bul M$ can be thought of as those ``compatible pairs'' for which the action is well-defined. For every $m\in M$, put $X_m\coloneqq \left\{x \in X\mid (x,m)\in X\bul M\right\}$ and $\alpha_m:X_m\to X, x\mapsto \rho_X(x,m)$. The set $X_m$ is the domain for the action by the element $m$. For the sake of simplicity, we will often write $x\cdot m \coloneqq \alpha_m(x) = \rho_X(x,m)$. We can now consider the following pullbacks:
\begin{equation}\label{eq:pushcan2}
\begin{gathered}
\xymatrix @!0 @R=30pt @C=53pt{
& X\times M \\
(X\bul M)\times M \ar[ur]^(0.45){\rho_X\times M} & & X\bul M \ar@{ >->}[ul]_(0.45){\pi_X} \\
&(X\bul M)\bul M\ar@{}[u]|<<<{\pushout} \ar[ur]_(0.6){\rho_X\bul M} \ar@{ >->}[ul]^(0.6){\pi_{X\bul M}}
}
\end{gathered}
\,\ \text{ and } \,\
\begin{gathered}
\xymatrix @!0 @R=30pt @C=53pt{
 & X\times M & \\
(X\bul M) \times M \ar[ur]^(0.4){(X\times \Delta)\circ(\pi_X \times M)\quad } & & X\bul M \ar@{ >->}[ul]_-{\pi_X} \\
& X\bul(M\bul M) \ar@{}[u]|<<<{\pushout} \ar@{ >->}[ul]^(0.6){\pi_{X,\Delta}} \ar[ur]_(0.6){X\bul \Delta} &
}
\end{gathered}
\end{equation}
Explicitly, these pullbacks can be described as the following sets:
\begin{align*}
(X\bul M)\bul M & = \left\{(x,m,n)\in X\bul M\times M\mid (x\cdot m,n)\in X\bul M\right\} \notag \\
 & = \left\{(x,m,n)\in X\times M\times M\mid x\in X_m\ \text{and}\ x\cdot m \in X_n\right\}, \\
X\bul (M\bul M) & = \left\{(x,m,n)\in X\bul M\times M\mid (x,mn)\in X\bul M\right\} \\ 
 & = \left\{(x,m,n)\in X\times M\times M\mid x\in X_m\ \text{and}\ x\in X_{mn}\right\}. 
\end{align*}
The quadruple $(X,X\bul M,\pi_X,\rho_X)$ is called a {\em partial action} of $M$ on $X$ if the following two axioms are satisfied.
\begin{enumerate}[label=({PA\arabic*}),leftmargin=1.1cm]
\item\label{item:QPC1} Unitality: 
$X_e=X$ and $\alpha_e=\id_X$.
Equivalently, 
there exists a morphism $X\bul u:X\to X\bul M$ which makes the following diagram commutative
\begin{equation}\label{eq:counitstrict}
\begin{gathered}
\xymatrix @!0 @R=30pt @C=65pt{
X \times M & X\bul M \ar@{<-}[d]_-{X\bul u} \ar@{ >->}[l]_-{\pi_X} \ar[r]^-{\rho_X} & X \\
& X. \ar[ul]^-{X\times u} \ar[ur]_-{\id_X}
}
\end{gathered}
\end{equation}
\item\label{item:QPC2} Partial associativity: 
$\alpha_m^{-1}(X_n) = X_{mn}\cap X_m$ and $\alpha_n\circ\alpha_m = \alpha_{mn}$ on $\alpha_m^{-1}(X_n)$ for all $m,n\in M$.
Equivalently, there is an isomorphism (equality, in fact)
\[\theta:(X\bul M)\bul M\to X\bul(M\bul M)\] 
such that the following diagram commutes
\begin{equation}\label{eq:geomcoass}
\begin{gathered}
\xymatrix @!0 @R=35pt @C=47pt{
X \ar@{<-}[rr]^-{\rho_X} \ar@{<-}[d]_-{\rho_X} & & X\bul M \ar@{<-}[rr]^-{\rho_X\bul M} & & (X\bul M)\bul M \ar[dll]_-{\theta} \ar@{ >->}[d]^-{\pi_{X\bul M}}
 \\
X\bul M & & X\bul(M\bul M) \ar[ll]^-{X\bul \Delta} \ar@{ >->}[rr]_{\pi_{X,\Delta}} & & (X\bul M)\times M. 
}
\end{gathered}
\end{equation}
\end{enumerate}

\begin{remark}
One can easily verify that the definition of partial action of a monoid as given above is equivalent to those considered for example in \cite[Definition 2.3]{Megrelishvili} and in \cite[Definition 2.4]{Hollings}. Moreover, it was shown in \cite[\S1]{JoostJiawey} that, for $M$ a group, one recovers the definition of partial group actions as given in \cite[Definition 1.2]{Exel-set}.

The above definitions of partial action datum and partial action can obviously be extended to the setting of monoids in arbitrary monoidal categories with pullbacks. Indeed, if $(M,\Delta,u)$ is a monoid (or algebra) in the monoidal category $(\cC,\otimes,{\mathbb I})$, then we can define {\em partial module data} $(X,X\bul M,\pi_X,\rho_X)$ and {\em geometric partial modules} by simply replacing the cartesian product by the monoidal product in diagrams \eqref{datum}, \eqref{eq:pushcan2}, \eqref{eq:counitstrict} and \eqref{eq:geomcoass} above.
This is the viewpoint of \cite{JoostJiawey}, where the dual notions of {\em partial comodule datum} and {\em geometric partial comodule} over a comonoid (coalgebra) in an arbitrary monoidal category with pushouts $\cC$ were defined. 

If $(X,X\bul M,\pi_X,\rho_X)$ and $(Y,Y\bul M,\pi_Y,\rho_Y)$ are geometric partial modules , then a \emph{morphism of geometric partial modules} is a pair $(f,f\bul M)$ of morphisms in $\cC$ with $f:X\to Y$ and $f\bul M: X\bul M \to Y\bul M$ such that the following diagram commutes
\begin{equation}\label{eq:shield}
\begin{gathered}
\xymatrix @!0 @R=30pt @C=60pt {
X \ar[d]_-{f} \ar@{<-}[r]^-{\rho_X} & X\bul M \ar[d]_-{f\bul M} \ar@{ >->}[r]^-{\pi_X} & X\ot M \ar[d]^-{f\ot M} \\
Y \ar@{<-}[r]_-{\rho_Y} & Y \bul M \ar@{ >->}[r]_-{\pi_{Y}} & Y \ot M. 
}
\end{gathered}
\end{equation}
We will often denote a geometric partial module $(X,X\bul M,\pi_X,\rho_X)$ simply by $X$ and a morphism as above simply by $f$. Moreover, we will denote by $\gpmod{M}$ the category of geometric partial modules over $M$ and their morphisms and we will often omit to specify the adjective ``geometric'' when not needed.

Note also that any usual (global) $M$-module $(X, \delta_X)$ is a geometric partial module with $\pi_X\coloneqq \id_{X \ot M}$ and $\rho_X \coloneqq \delta_X$. In fact, $\Rmod{M}$ is a full subcategory of $\gpmod{M}$ and we denote by $\cI: \Rmod{M} \to \gpmod{M}$ the associated embedding functor.
\end{remark}

\subsection{The general globalization result}\label{ssec:globgen}

Let $(M,\Delta,u)$ be a monoid in a monoidal category $\cC$ with pullbacks.
Recall from (the dual of) \cite[Example 2.5]{JoostJiawey} that, for any (right) $M$-module $(Y,\delta)$ and any monomorphism $p : X \to Y$ in $\cC$, the pullback
\begin{equation}\label{eq:globcom}
\begin{gathered}
\xymatrix @!0 @R=18pt @C=57pt {
 & Y & \\
 X\otimes M \ar[ur]^(0.45){(p\otimes M)\circ\delta\ } & & X \ar@{ >->}[ul]_-{p} \\
 & X\bul M \ar@{}[u]|<<<{\pushout} \ar@{ >->}[ul]^-{\pi_X} \ar[ur]_-{\rho_X} &
}
\end{gathered}
\end{equation}
inherits a structure of geometric partial module and $p$ becomes a morphism of partial modules. We refer to this as the \emph{induced partial module} structure from $Y$ to $X$.

Naively speaking, the globalization of a partial module $X$ is a universal  $M$-module ``containing'' $X$ and such that the partial action is induced by the global one. The following definition is the straightforward dualization of \cite[Definition 3.11]{Saracco-Vercruysse}.

\begin{definition}\label{def:glob}
Given a partial module $(X,X\bul M,\pi_X,\rho_X)$, a \emph{globalization} for $X$ is a global module $(Y,\delta_Y)$ with a morphism $p:X\to Y$ in $\cC$ such that
\begin{enumerate}[label=(GL\arabic*),ref=(GL\arabic*),leftmargin=1.1cm]
\item\label{item:GL0} the following diagram commutes (i.e., $p$ is a morphism of partial modules) and it is a pullback diagram
\[
\begin{gathered}
\xymatrix @R=15pt{
X\ot M \ar[r]^-{p\ot M} & Y\ot M \ar[r]^-{\delta_Y} & Y \\
X\bul M \ar[u]^-{\pi_X} \ar[rr]_-{\rho_X} & & X; \ar[u]_p
}
\end{gathered}
\]
\item\label{item:GL3} it is universal, in the sense that the following correspondence is bijective
\[\Rmod{M}(Y,Z) \to \gpmod{M}(X,\cI(Z)), \qquad \eta \mapsto \eta \circ p.\]
\end{enumerate}
We say that $X$ is {\em globalizable} if a globalization for $X$ exists and we denote by $\gpmod{M}^{gl}$ the full subcategory of $\gpmod{M}$ of the globalizable partial modules. 
\end{definition}

It can be shown (see \cite[Lemma 3.2]{Saracco-Vercruysse}) that if $(Y,p)$ is a globalization of a partial module $X$, then $p:Y\to X$ is a monomorphism. 
Moreover, it follows from axiom \ref{item:GL3} that a globalization of a partial module is unique whenever it exists.

The following theorem is the main results of \cite{Saracco-Vercruysse}, Theorem 3.5, rephrased in its dual form for the sake of the reader.

\begin{theorem}\label{thm:globalization}
Let $M$ be a monoid in the monoidal category $\cC$ with pullbacks. Then a geometric partial $M$-module $X=(X,X\bul M,\pi_X,\rho_X)$ is globalizable if and only if 
\begin{enumerate}[label=(\alph*),ref={\itshape(\alph*)},leftmargin=0.6cm]
\item\label{item:glob1} the following coequalizer exists in $\mod{M}$:
\begin{equation}\label{eq:glob}
\begin{gathered}
\xymatrix@C=30pt{
(X\bul M\otimes M, X\bul M\ot \Delta) \ar@<+0.8ex>[rr]^-{\rho_X\ot M} \ar@<-0.4ex>[rr]_-{(\pi_X\ot M)\circ(X \otimes \Delta)} & &  (X\otimes M,X\ot \Delta) \ar@<+0.2ex>[r]^-{\kappa} & (Y_X,\delta);
}
\end{gathered}
\end{equation}
\item\label{item:glob3} the following diagram is a pullback diagram in $\cC$:
\begin{equation}\label{eq:GXglob}
\begin{gathered}
\xymatrix @!0 @R=18pt @C=60pt {
 & Y_X & \\
 X \otimes M \ar[ur]^(0.45){\kappa} & & X \ar[ul]_(0.4){\kappa\circ (X\ot u)} \\
 & X\bul M. \ar@{ >->}[ul]^-{\pi_X} \ar[ur]_-{\rho_X} &
}
\end{gathered}
\end{equation}
\end{enumerate}
Moreover, if these conditions hold, then the morphism $\epsilon_X \coloneqq \kappa\circ (X\ot u):X\to Y_X$ is a monomorphism in $\cC$,
$\kappa = \delta\circ (\epsilon_X \otimes M)$ and $(Y_X,\epsilon_X)$ is the globalization of $X$.
\end{theorem}

\subsection{Recovering the globalization of partial actions of groups and monoids}

Let us return to the situation where the monoid $M$ (in $\Set$) acts partially on the set $X$. 
Then the following coequalizer (in $\Set$)
\begin{equation}\label{eq:coeq}
\xymatrix @C=30pt {
X\bul M \times M \ar@<+0.5ex>[rr]^-{\rho_X\times M} \ar@<-0.5ex>[rr]_-{(X\times \Delta) \circ (\pi_X\times M)} & & X\times M \ar[r]^-{\kappa} & Y_X
}
\end{equation}
is given by $Y_X=(X\times M)/R$, where $R\subseteq (X\times M)\times (X\times M)$ is the equivalence relation generated by $r = \big\{\big((x\cdot m,n),(x,mn)\big)\mid m,n\in M, x\in X_m\big\}$. Since the endofunctor $-\times M:\Set\to\Set$ is a left adjoint, it preserves coequalizers and hence $Y_X$ inherits in a natural way a global action from $X\times M$. Explicitly, if $[x,m]$ denotes the class of $(x,m)$ in $Y_X$, then the global action of $M$ on $Y_X$ is given by $[x,m]\triangleleft n \coloneqq [x,mn]$.
Applying our globalization Theorem \ref{thm:globalization}, we find that $Y_X$ will be the globalization of $X$ if \eqref{eq:GXglob} is a pullback diagram. This was essentially proven in \cite[Proposition 2.6]{Megrelishvili}. Hence we can conclude the following result, which then shows that the globalization for monoids as described in \cite[\S2]{Megrelishvili} is a special instance of the globalization for geometric partial modules discussed in the previous section.

\begin{corollary}\label{cor:Set}
For $\cC = \Set$, we have $\gpmod{M}^{gl} = \gpmod{M}$ for every monoid $M$.
\end{corollary}

Let us remark that if the monoid $M$ is a group, then the globalization $Y_X$ coincides with the globalization for partial group actions as given in \cite[Theorem 1.1]{Abadie} or \cite[\S3.1]{KellendorkLawason}, as already discussed in \cite[Proposition 3.4]{Saracco-Vercruysse}. 

\section{Partial actions of topological monoids}\label{ssec:top}

Consider the category $\Top$
of topological spaces. It is a monoidal, complete and cocomplete category (see \eg \cite[Chapter V, \S9]{MacLane}). 
Limits and colimits can be computed by endowing the corresponding limits and colimits in $\Set$ with a suitable topology.  
%
A monoid in $\Top$ is a topological monoid $((M,\top_M),\Delta,u)$, i.e.\ a topological space endowed with a monoid structure whose composition 
is a continuous map. The notion of geometric partial module inflected in $\Top$ gives a span
\[
\xymatrix @!0 @C=75pt @R=25pt {
(X\times M,\top_X \times \top_M) & & (X,\top_X) \\
 & (X\bul M,\top_{X \bul M}) \ar@{ >->}[ul]^-{\pi_X} \ar[ur]_-{\rho_X} & 
}
\]
in $\Top$, where $\pi_X$ is an injective continuous map and \ref{item:QPC1} and \ref{item:QPC2} hold. 
The following example shows that, in contrast to what we saw in the previous section for $\Set$, not every geometric partial module over a topological monoid is globalizable.

\begin{example}\label{ex:countertop}
Let $(M,\Delta,u)$ be a topological monoid
and $X$ a set with at least two elements, which we endow with the indiscrete topology. Consider the trivial global action $\delta_X:X\times M\to X$, given by $\delta_X(x,m)=x$ for all $x\in X$ and $m\in M$. 
Now we define the topological space $X\bul M$ as the set $X\times M$ endowed with the product topology of the discrete topology $\top_X^{\mf{d}}$ on $X$ and the given topology on $M$. In other words, the topology on $X\bul M$ is generated by open sets of the form $\{x\}\times U$ where $x\in X$ and $U\in \top_M$.
Then $(X,X\bul M,\id_{X\times M},\delta_X)$ is clearly a partial $M$-module datum in $\Top$. Since the action is global, both $(X\bul M)\bul M$ and $X\bul (M\bul M)$ have $X\times M\times M$ as underlying set. On $(X\bul M)\times M$ we have the product topology arising from the discrete topology on $X$ and the given topology on $M$ and since 
\[
\begin{gathered}
\delta_X \bul M = \big(X \times M \times M, \top_X^{\mf{d}} \times \top_M \times \top_M\big) \xrightarrow{\delta_X \times M} \big(X \times M, \top_X^{\mf{d}} \times \top_M\big) \quad\text{and} \\
X\bul \Delta = \big(X \times M \times M, \top_X^{\mf{d}} \times \top_M \times \top_M\big) \xrightarrow{X \times \Delta} \big(X \times M, \top_X^{\mf{d}} \times \top_M\big)
\end{gathered}
\] 
are already continuous, we find that $(X\bul M)\bul M= X\bul (M\bul M)$ have the same topology as $(X\bul M)\times M$ and hence $X$ is a geometric partial $M$-module.

The underlying set of the coequalizer \eqref{eq:coeq} in $\Top$ is $X$ (because the original action was global) endowed with the quotient topology along $X \times M \xrightarrow{\delta_X} X$, which is the original indiscrete topology on $X$. 
Moreover, if we endow this coequalizer $Y_X$ with the initial global action of $X$, one easily observes that $Y_X$ is also coequalizer \eqref{eq:glob} in the category of (global) topological $M$-modules $\mod{M}$. Hence condition \ref{item:glob1} of Theorem \ref{thm:globalization} is satisfied and therefore $(X,X\bul M,\id_{X\times M},\delta_X)$
is globalizable (with globalization the global $M$-module $(Y_X,\delta_X)$) if and only if \eqref{eq:GXglob} is a pullback in $\Top$. Specifying \eqref{eq:GXglob} in our present setting, we obtain the following diagram
\[
\xymatrix @!0 @C=55pt @R=18pt{
 & Y_X 
 & \\
X\times M 
\ar[ur]^-{\delta_X} & & X 
\ar[ul]_-{\id_X} \\
 & X\bul M 
 \ar[ul]^-{\id_{X \times M}} \ar[ur]_-{\delta_X} & 
}
\]
In order for this diagram to be a pullback in $\Top$, $X\bul M$ should have the coarsest topology for which $\id_{X\times M}$ and $\delta_X$ are continuous, which means that the opens should be of the form $X\times U$, with $U\subset M$ open. This clearly differs from the topology we have chosen on $X\bul M$. 
Therefore, the diagram above is not a pullback and hence the geometric partial module $(X,X\bul M,\id_{X\times M},\delta_X)$ is not globalizable.

\end{example}

As a consequence, we can state the following result.

\begin{proposition}
In the category $\Top$, a general globalization theorem for geometric partial modules does not exist. 
More precisely, $\gpmod{M}^{gl} \subsetneq \gpmod{M}$ for any topological monoid $M$.
\end{proposition}

\begin{proof}
It follows directly from Example \ref{ex:countertop}. 
\end{proof}

Although we know that not all geometric partial modules in $\Top$ are globalizable, Theorem \ref{thm:globalization} provides for us a way to describe globalizable partial modules over nice classes of topological monoids.

\begin{theorem}\label{thm:topmain}
Let $(M,\top_M)$ be a topological monoid for which the category of global $M$-modules in $\Top$ admits coequalizers (of type \eqref{eq:glob}) and the underlying functor to $\Top$ preserves them.

Then the globalizable geometric partial modules over $M$ in $\Top$ are exactly all those geometric partial modules $(X,X\bul M, \pi_X,\rho_X)$ for which $X \bul M$ has the coarsest topology making both $\pi_X$ and $\rho_X$ continuous. In this case, the globalization is given by the coequalizer \eqref{eq:coeq}.

In particular, any geometric partial module for which $\pi_X$ is an embedding (i.e., $X \bul M$ has the induced topology via $\pi_X$), is globalizable.
\end{theorem}

\begin{proof}
Since $\Top$ has a faithful forgetful functor to $\Set$ which preserves limits (so, in particular, pullbacks), any geometric partial module in $\Top$ over a topological monoid $M$ is also a geometric partial module in $\Set$ over the underlying monoid of $M$. From Corollary \ref{cor:Set}, we know that the latter is globalizable. 

Furthermore, by Theorem \ref{thm:globalization}, we know that whenever a globalization of a geometric partial module $X$ in $\Top$ exists, this globalization can be computed by means of the coequalizer \eqref{eq:glob} in the category of global $M$-modules. Under the assumption of the theorem, this coequalizer can already be computed in $\Top$, and as the forgetful functor $\Top\to\Set$ also preserves colimits (in particular, coequalizers), the underlying (set-theoretical) $M$-module of the (topological) globalization of $X$ coincides with the set-theoretical globalization of $X$. Therefore, the sole criterion for the globalization of $X$ to exist, is that the set-theoretical pullback \eqref{eq:GXglob} is also a pullback in $\Top$. Clearly, this is the case exactly if the topology on $X \bul M$ is the coarsest topology making both $\pi_X$ and $\rho_X$ continuous.

Finally, if $(X,X\bul M, \pi_X,\rho_X)$ is a geometric partial module for which $\pi_X$ is an embedding, then this means that $X\bul M$ has simply the induced (or subspace) topology from $X\times M$ and that $\rho_X$ is already continuous with respect to this topology. Hence $X\bul M$ has indeed the coarsest topology for which both $\pi_X$ and $\rho_X$ are continuous and we can conclude.
\end{proof}

The following Lemma shows that the condition on the topological monoid in Theorem \ref{thm:topmain} is satisfied in some very relevant cases.

\begin{lemma}\label{lem:niceM}
If $(M,\top_M)$ is a topological monoid such that one of the following conditions holds:
\begin{enumerate}[label=(\arabic*),ref=(\arabic*),leftmargin=0.7cm]
\item\label{item:niceM1} The endofunctor $- \times M \times M:\Top\to\Top$ preserves coequalizers;
\item\label{item:niceM2} $M$ is core-compact;
\item\label{item:niceM3} $(M,\top_M)$ is a topological group;
\end{enumerate}
then the category of global $M$-modules in $\Top$ admits coequalizers and the underlying functor to $\Top$ preserves them.
\end{lemma}

\begin{proof}
\underline{\ref{item:niceM1}}. This is an application of the well-known fact that for any monoid $M$ in a monoidal category $\cC$, the colimit of any given diagram in the category $\mod{M}$ exists whenever the colimit of the same diagram in $\cC$ exists and the functor $-\ot M\ot M : \cC \to \cC$ preserves it.\\
\underline{\ref{item:niceM2}}. This is a particular instance of \ref{item:niceM1}, since in this case $- \times M$ is a left adjoint functor (see, for instance, \cite[Theorem 5.3]{Escardo}).\\
\underline{\ref{item:niceM3}}. Also this is a particular instance of \ref{item:niceM1}, but specified to the coequalizers of type \eqref{eq:coeq} (which is sufficient for Theorem \ref{thm:topmain} to hold). Recall that the endofunctor $- \times Z$ preserves coequalizers $(Q,q)$ of open maps in $\Top$, for every $Z$ in $\Top$: in fact, being the coequalizer $(Q,q)$ of open maps open itself, 
$q \times Z$ is open, 
surjective and continuous and hence it is a quotient map (that is, the product topology on $Q \times Z$ is equivalent to the quotient topology).
Now, being $G$ a group, the maps $\Delta:G \times G \to G$, $\pi_X$ and $\rho_X$ are all open maps. 
Since products and coequalizers of open maps are open again and the endofunctor $- \times G$ preserves coequalizers of open maps, $- \times G$ always preserves the coequalizer \eqref{eq:coeq} in $\Top$.
\end{proof}


\noindent The following definition subsumes at the same time \cite[Definition 1.1]{Abadie} (in case $((M,\top_M),\Delta,u)$ is a topological group) and \cite[page 125]{Megrelishvili} (in case $M$ is discrete).


\begin{definition}\label{def:TopParAct}
A \emph{topological partial (right) action} of a topological monoid $(M,\top_M)$ on a topological space $(X,\top_X)$ is a pair $\left(\left\{X_m\right\}_{m\in M},\left\{\alpha_m\right\}_{m\in M}\right)$ such that
\begin{enumerate}[label = (TP\arabic*),ref = (TP\arabic*),leftmargin=1.1cm]
\item\label{item:TPA2} the set $X\bul M =\left\{(x,m)\in X\times M\mid x\in X_m\right\}$ is an open subspace of $X\times M$ and the function $\rho_X : X \bul M \to X, (x,m)\mapsto \alpha_m(x)$ is continuous; 
\item\label{item:TPA3} the pair forms a set-theoretical partial action of $M$: $X_e = X$, $\alpha_e = \id_X$ and for all $m,n\in M$, $\alpha_m^{-1}(X_n) = X_{mn}\cap X_m$ and $\alpha_n\circ\alpha_m = \alpha_{mn}$ on $\alpha_m^{-1}(X_n)$.
\end{enumerate}
A \emph{morphism of topological partial actions} from $(X,\top_X)$ with $\left(\left\{X_m\right\}_{m\in M},\left\{\alpha_m\right\}_{m\in M}\right)$ to $(X',\top_{X'})$ with $\left(\left\{X_m'\right\}_{m\in M},\left\{\alpha_m'\right\}_{m\in M}\right)$ is a continuous map $f : X \to X'$ such that $f(X_m)\subseteq X_m'$ and $f \circ \alpha_m = \alpha_m'\circ f$.
\end{definition}

Remark that axiom \ref{item:TPA2} implies that the set $X_m= \left\{x \in X\mid (x,m)\in X\bul M\right\}$ is an open subspace of $X$ and $\alpha_m:X_m\to X$ is a continuous map, for all $m\in M$. This was included as an additional axiom in \cite[Definition 1.1]{Abadie} and \cite[page 125]{Megrelishvili}. 

Any discrete partial action as in \S\ref{se:partact} is an example of a topological partial action in which every space has the discrete topology. On the other hand, not every geometric partial module in $\Top$ is a topological partial action. Indeed, the Example \ref{ex:countertop} is a geometric partial module in $\Top$ which is not a topological partial action. This can be seen easily from the following proposition, since in Example \ref{ex:countertop} $\pi_X:X\bul M\to X\times M$ is an injective continuous map, but not an open embedding, as $X\bul M$ does not have the induced topology along $\pi_X$.

\begin{proposition}\label{proposition:TPA}
Topological partial actions of a topological monoid $(M,\top_M)$ are exactly the geometric partial $M$-modules in ${\Top}$ for which $\pi_X$ is an open embedding.

Furthermore, the category $\ms{TopParAct}_M$ of topological partial actions of $M$ and their morphisms is a full subcategory of the category $\gpmod{M}$ of geometric partial $M$-modules in ${\Top}$.
\end{proposition}

\begin{proof}

As explained in \S\ref{se:partact}, axiom \ref{item:TPA3} tells exactly that $(X,X\bul M,\pi_X,\rho_X)$ is a geometric partial module over $M$ in $\Set$. Furthermore, axiom \ref{item:TPA2} tells that $\pi_X$ and $\rho_X$ are morphisms in $\Top$ and  $\pi_X$ is an open embedding. Hence we can conclude on the first assertion of the theorem if we prove that the bijection 
$\theta:(X \bul M) \bul M \to X \bul (M \bul M)$ (arising from the fact that $(X,X\bul M,\pi_X,\rho_X)$ is a geometric partial module in $\Set$) is an homeomorphism.
The condition $X_m \cap X_{mn} = \alpha_m^{-1}(X_n)$ implies that the continuous map
\[
X \bul (M \bul M) \xrightarrow{\pi_{X,\Delta}} X \bul M \times M \xrightarrow{\rho_X \times M} X \times M
\]
factors through $(X \bul M, \pi_X)$ and since the latter has the induced topology from $X\times M$, the resulting factorization is continuous. Such a factorization is exactly the map needed to prove that the inclusion $X \bul (M \bul M) \to (X \bul M) \bul M$ is continuous by resorting to the universal property of $(X \bul M) \bul M$ as a pullback in $\Top$. The other way around, the argument is similar. 
%
%

Now, any 
morphism $f:X\to X'$ of topological partial actions induces a function $f \bul M : X \bul M \to X'\bul M$ by (co)restriction of $f \times M$, which is continuous and makes \eqref{eq:shield} to commute. Moreover, if $(f,f\bul M)$ is a morphism of partial $M$-modules which were induced by topological partial actions, then the condition $(f \times M) \circ \pi_X = \pi_{X'} \circ (f \bul M)$ entails that for every $x \in X_m$, we have $f(x) \in X'_m$, and the condition $f \circ \rho_{X} = \rho_{X'} \circ (f \bul M)$ entails that for every $x \in X_m$, we have $f(\alpha_m(x)) = \alpha'_m(f(x))$. Therefore, $f$ is a morphism of topological partial actions. One easily verifies that this construction is functorial. 
\end{proof}

By combining Proposition \ref{proposition:TPA} with Theorem \ref{thm:topmain}, we immediately can conclude that all topological partial actions over a `nice' topological monoid are globalizable.

\begin{corollary}\label{topactisglob}
Let $M$ be a topological monoid as in Theorem \ref{thm:topmain} (e.g. $M$ satisfies one of the conditions from Lemma \ref{lem:niceM}), then every topological partial action over $M$ is globalizable and the globalization is given by the coequalizer \eqref{eq:coeq}. Hence $\ms{TopParAct}_M$ is a full subcategory of $\gpmod{M}^{gl}$.\end{corollary}

Let us remark that in the framework of topological partial actions of topological groups, Corollary \ref{topactisglob} was also proven in \cite[Proposition 5.5]{ExelBig} and \cite[Theorem 1.1]{Abadie}. In the framework of topological partial actions of a (discrete) monoid on a topological space, an analogous of this result has been established in \cite[\S3]{Megrelishvili}.

In general, however, the inclusion of $\ms{TopParAct}_M$ in $\gpmod{M}^{gl}$ in Corollary \ref{topactisglob} is not essentially surjective on objects, in the sense that there exist globalizable partial modules which are not coming from topological partial actions. 

\begin{example}\label{ggpmnotparact}
Take $M = \R$ acting on $Y = \R^2$ by vertical translation $Y \times M \to Y, \big((x,y),v\big) \mapsto (x,y+v)$ and take $X$ to be the subspace $j: \R \to \R^2, x \mapsto (x,0),$ (everything with Euclidean topology). Then $X \bul M = \R \times \{0\}$, which is not open in $\R \times \R$. However, being $(X,X\bul M,\pi_X,\rho_X)$ the trivial partial module, it is globalizable with globalization $(X \times M, X \times \Delta)$ (see \cite[Proposition 3.10]{Saracco-Vercruysse}).
\end{example}

The next proposition further explains this phenomenon and how topological partial actions can be characterized by the way they embed in their globalization.




\begin{proposition}\label{globparact}
Let $M$ be a topological monoid.

If $(Y,\delta)$ is a global $M$-module in $\Top$ and $\epsilon:X\to Y$ is an open embedding,
then the induced geometric partial $M$-module $(X,X\bul M,\pi_X,\rho_X)$ in $\Top$ obtained by restricting $Y$ along $\epsilon$ is a topological partial action of $M$.

Conversely, if $X$ is a topological partial action of $M$ which is globalizable (as geometric partial module) with globalization $Y_X$, then the monomorphism $\epsilon_X:X\to Y_X$ is an open embedding.
\end{proposition}

\begin{proof}
Let $(Y,\delta)$ be a global topological $M$-module and $\epsilon:X\to Y$ be an open embedding. Then one can endow $X$ with a geometric partial comodule structure by taking the pullback \eqref{eq:globcom}. Hence we can identify $X\bul H=\delta^{-1}(\epsilon(X))\cap (X\times M)$. Since $\epsilon$ is an open embedding and $\delta$ is continuous, we can conclude that $X\bul H$ is an open subset of $X\times M$ and $\rho_X$ (which is the restriction of $\delta$) is continuous with respect to the subset topology on $X\bul H$.
In other words $\pi_X:X\bul M\to X\times M$ is an open embedding and hence $X$ is a topological partial action by Proposition\ref{proposition:TPA}. 

Suppose now that $X$ is a globalizable topological partial action with globalization $(Y_X,\delta)$. 
Then the topology $\top_{X \bul M}$ on $X \bul M$ has to be the limit topology and hence the coarsest topology for which $\pi_X$ and $\rho_X$ are continuous. Moreover, since $\big((Y_X,\top_Y),\kappa,\big)$ is (up to homeomorphism) the coequalizer in $\Top$ of
\[
\xymatrix @C=65pt{ 
\big(X \bul M \times M, \top_{X \bul M} \times\top_M \big) \ar@<+0.5ex>[r]^-{\rho_X \times M} \ar@<-0.5ex>[r]_-{(X \times \Delta)\circ (\pi_X \times M)} & \big(X \times M, \top_X \times \top_M\big),
 }
\]
we have that $\top_Y$ is the quotient topology with respect to $\kappa$, that is $U \in \top_Y$ if and only if $\kappa^{-1}(U) \in \top_X \times \top_M$, for every $U \subseteq Y_X$. In particular, $\epsilon_X(X) \in \top_Y$ if and only if $X \bul M = \kappa^{-1}\big(\epsilon_X(X)\big) \in \top_X \times \top_M$. Therefore, as $X \bul M$ is open in $X \times M$ (see Proposition \ref{proposition:TPA}), also $\epsilon_X(X)$ is open in $Y_X$ (and conversely). 

We are left to check that $\epsilon_X$ is an open map. For every $V \in \top_X$, we have that
\[
\kappa^{-1}\big(\epsilon_X(V)\big) = \big\{(x,m) \in X \times M \mid [x,m] = [y,e] \text{ for some } y\in V\big\}.
\]
However, by definition of the relation $R$ defining $Y_X$, the condition $[x,m] = [y,e]$ implies that $x\cdot m$ exists (i.e.~that $(x,m) \in X \bul M$) and $x \cdot m = y \in V$. 
Therefore, $\kappa^{-1}\big(\epsilon_X(V)\big) \subseteq \pi_X\big(\rho_X^{-1}(V)\big)$. Since also the inclusion in the opposite direction is true, we have that $\kappa^{-1}\big(\epsilon_X(V)\big) = \pi_X\big(\rho_X^{-1}(V)\big) \in \top_X \times \top_M$ and so $\epsilon_X(V) \in \top_Y$.
\end{proof}

\begin{corollary}
Let $M$ be a topological monoid as in Theorem \ref{thm:topmain} (e.g. $M$ satisfies one of the conditions from Lemma \ref{lem:niceM}). Then, topological partial actions over $M$ are exactly those globalizable geometric partial modules that embed in their globalization as open subspaces.
\end{corollary}

%

Let $G$ be a topological group, $X$ a topological partial action and $Y_X$ the globalization of $X$, which we know that exists from the above and which is constructed as a suitable quotient $G\times X/\sim$ (since it is the coequalizer \eqref{eq:coeq}). 
Abadie observed in \cite{Abadie} that, for $G$ a Hausdorff topological group and $X$ a Hausdorff topological partial action, this quotient is not necessarily Hausdorff in general. Consequently, this quotient is not (isomorphic to) the coequalizer $Y_X$ from \eqref{eq:coeq} in the category ${\Haus}$. 
However, since ${\Haus}$ is still complete and cocomplete (see \eg \cite[Proposition V.9.2]{MacLane}), one can still consider the coequalizer $(Y_X,\kappa)$ of
\begin{equation}\label{eq:coeqTop}
\xymatrix @C=35pt{
X\bul G \times G \ar@<+0.5ex>[rr]^-{\rho_X\times G} \ar@<-0.5ex>[rr]_-{(X\times \Delta)\circ(\pi_X\times G)} & & X\times G 
}
\end{equation}
in $\Haus$, which is the ``largest Hausdorff quotient'' of the coequalizer $(Y',\kappa')$ of the same pair of arrows, but computed in $\Top$. Namely, $Y_X \coloneqq Y'/\approx$ where $y \approx y'$ if and only if for every $f:Y' \to Q$ with $Q$ in $\Haus$ we have $f(y)=f(y')$ (i.e., they cannot be distinguished by maps to Hausdorff spaces). As a matter of notation, denote by $[x,m]$ the equivalence class of $(x,m)$ in $Y_X$, by $[x,m]'$ its class in $Y'$ and by $q:Y' \to Y_X$ the canonical projection. We have $q\circ \kappa' = \kappa$.

The following result tells that the globalization for Hausdorff partial actions exists exactly when the coequalizer of \eqref{eq:coeqTop} can be computed as in $\Top$. This should be compared to \cite[Proposition 1.2]{Abadie} and \cite[Proposition 5.6]{ExelBig}.

\begin{theorem}\label{thm:Haus}
Consider a geometric partial module $(X,X\bul G,\pi_X,\rho_X)$ in ${\Haus}$ where $G$ is a group. Then $X$ is globalizable with globalization $Y_X$ if and only if 
$Y_X=Y'$, that is, if and only if 
the coequalizer of \eqref{eq:coeqTop} in $\Top$ is a Hausdorff space.
\end{theorem}


\begin{proof}
The reverse implication holds in light of Corollary \ref{topactisglob}, whence let us focus on the direct one. Assume then that $Y_X$ is the globalization of $X$ with global action $\beta$ and pick two distinct points $[x,g]',[y,h]'$ in $Y'$ (where $Y'$ as above denotes the coequalizer of the pair \eqref{eq:coeqTop} in $\Top$). Consider $[x,gh^{-1}]'$ and $[y,e]'$. If $[x,gh^{-1}]=[y,e]$, then $(x,gh^{-1})\in X\bul G$ and so $[x,g]'=[y,h]'$, a contradiction. Thus, $[x,gh^{-1}]\neq [y,e]$ and hence there exists $Q$ Hausdorff and $f:Y' \to Q$ such that $f([x,gh^{-1}]')\neq f([y,e]')$. By taking the preimages of two separating open sets, we find two open subsets $U,V$ of $Y'$ separating $[x,gh^{-1}]'$ from $[y,e]'$. Since $\beta_h$ is an homeomorphism, $\beta_h(U)$ and $\beta_h(V)$ are open subsets separating $[x,g]'$ from $[y,h]'$ in $Y'$.
\end{proof}

%


\section{Partial comodule algebras}\label{ssec:alg}


Let $\K$ be a commutative ring. Recall that the category $\algk$ of (unital, associative) $\K$-algebras is monoidal, the monoidal product being the tensor product of two $\K$-algebras with component-wise multiplication. In this section, we study geometric partial modules (see \S\ref{se:partact}) in the monoidal category $\op{\algk}$ or, stated otherwise, geometric partial comodules in $\algk$ (by using the dual terminology from \cite{JoostJiawey}), which we will call {\em geometric partial comodule algebras}.

Firstly, recall that a coalgebra in $\algk$ is just a $\K$-bialgebra $(H,\mu,u,\Delta,\varepsilon)$ and a global comodule over $H$ in $\algk$ is exactly a $H$-comodule algebra in the classical sense. A geometric partial comodule algebra is a quadruple $(A,A\bul H,\pi_A,\rho_A)$, where $A$ and $A\bul H$ are algebras, $\pi_A:A\ot H\to A\bul H$ is an algebra epimorphism (not necessarily surjective) and $\rho_A:A\to A\bul H$ is an algebra morphism, satisfying the counitality and coassociativity conditions dual to axioms \ref{item:QPC1} and \ref{item:QPC2}. 
Specializing Theorem \ref{thm:globalization} to this setting, we find the following result.

\begin{corollary}\label{comodalgglob}
Let $H$ be a $\K$-bialgebra which is flat as left $\K$-module and consider a geometric partial $H$-comodule algebra $(A,A\bul H,\pi_A,\rho_A)$. Set
$$Y\coloneqq\left\{\sum_i a_i\ot h_i\in A\ot H ~\Big|~ \sum_i \rho_A(a_i)\ot h_i=\sum_i \pi_A(a_i\ot h_{i(1)})\ot h_{i(2)}\right\}.$$
Then the globalization of $A$ exists 
provided that the following diagram
\begin{equation}\label{eq:pcomalgs}
\begin{gathered}
\xymatrix @!0 @C=50pt @R=17pt{
& Y \ar[dl]_-{A\ot\varepsilon} \ar@{^(->}[dr] \\
A \ar[dr]_-{\rho_A} && A\ot H \ar[dl]^-{\pi_A}\\
& A\bul H
}
\end{gathered}
\end{equation}
is a pushout in $\algk$. In this case, the globalization is given precisely by $Y$.

Moreover, in case $A\ot\varepsilon:Y\to A$ is surjective, then the above condition is satisfied provided that $\pi_A$ is surjective as well and $\ker(\pi_A)$ can be generated (as an $A\ot H$-ideal) by elements of $Y$. 
\end{corollary}

\begin{proof}
We know from Theorem \ref{thm:globalization} that the globalization of $A$, if it exists, should be given by the equalizer 
\begin{equation}\label{globcomalg}
\begin{gathered}
\xymatrix @C=25pt{
Y\ar[r]^-\kappa & A\ot H \ar@<.5ex>[rr]^-{\rho_A\ot H} \ar@<-.5ex>[rr]_-{(\pi_A\ot H)\circ(A\ot \Delta)} && A\bul H\ot H
}
\end{gathered}
\end{equation}
computed in $\algk^{H}$ (the $H$-comodules in $\algk$). Since $H$ is flat as a left $\K$-module, this equalizer can be computed in $\algk$, and hence in $\Rmod{\K}$ (or even $\Set$). Therefore, it is given by the set $Y$ in the statement. The second condition of Theorem \ref{thm:globalization} states exactly that diagram \eqref{eq:pcomalgs} is a pushout square in $\algk$. 

For the last statement, recall that the pushout of a span $R \xleftarrow{f} S \xrightarrow{g} T$ in $\algk$ where $g$ is a surjective map is given by $\left(R/\langle f(\ker(g)) \rangle,p_R,\tilde{f}\right)$
where $\langle f(\ker(g)) \rangle$ is the two-sided ideal in $R$ generated by $f(\ker(g))$, $p_R$ is the canonical projection and $\tilde{f}$ is the unique map such that $\tilde{f}\circ g = p_R \circ f$. By applying this, we see that the diagram in the statement is indeed a pushout if and only if $\pi_A$ is surjective and $\ker(\pi_A)$ is the ideal generated by all $\sum_i a_i\ot h_i\in Y$ such that $\sum_ia_i\varepsilon(h_i)=0$. 
\end{proof}

Let us remark that, in general, $\pi_A$ is not necessarily surjective and geometric partial comodule algebras are not always globalizable (see \cite[Example 3.6]{Saracco-Vercruysse} for an explicit example). 
We will now describe a particular class of geometric partial comodule algebras for which the globalization always exists. 

\begin{definition}[{\cite{Caenepeel-Janssen}}]\label{def:CJ}
A (right) \emph{algebraic\footnote{the prefix `algebraic' is not standard in literature, but we use it here to distinguish these objects from geometric partial comodule algebras as introduced above.} partial comodule algebra} over a bialgebra $H$ is an algebra $A$ with a $\K$-linear map $\delta_A:A\to A\ot H, a\mapsto a_{[0]}\ot a_{[1]},$ such that
\begin{enumerate}[label=(\roman*),ref=(\roman*),leftmargin=0.75cm]
\item $\delta_A(ab) = \delta_A(a)\delta_A(b)$, \ie $(ab)_{[0]} \ot (ab)_{[1]} = a_{[0]}b_{[0]} \ot a_{[1]}b_{[1]};$
\item\label{item:pca2} $(A\ot \varepsilon)\delta_A(a) = a$;
\item $\left(\delta_A \ot H\right)\delta_A(a) = \left(\delta_A(1_A)\ot H\right)\cdot \left(A\ot \Delta\right)\delta_A(a)$, \ie
\begin{equation}\label{eq:coass}
a_{[0][0]} \ot a_{[0][1]} \ot a_{[1]} = 1_{A[0]}a_{[0]} \ot 1_{A[1]}a_{[1](1)}  \ot a_{[1](2)}.\footnote{In \cite{Caenepeel-Janssen}, $1_{A[0]}\otimes 1_{A[1]}$ appears on the right: $a_{[0][0]} \ot a_{[0][1]} \ot a_{[1]} = a_{[0]}1_{A[0]} \ot a_{[1](1)}1_{A[1]}  \ot a_{[1](2)}$. Here we resort to the convention used in \cite{AlvesBatista}, for the sake of consistency with what follows. This change of side is harmless.}
\end{equation}
\end{enumerate}
\end{definition}

It has been shown in \cite[Example 4.9]{JoostJiawey} that any partial comodule algebra over $H$ in the sense of Definition \ref{def:CJ} is a geometric partial comodule in the category $\algk$. Briefly, set 
$e':=1_A\ot 1_H-1_{A[0]}\ot 1_{A[1]}$, which is an idempotent.
Consider the canonical projection
\begin{equation}\label{eq:piA}
\pi_A 
:A\ot H \to \frac{A\ot H}{\langle e'\rangle}.
\end{equation}
Setting $A\bul H:= A\ot H/\langle e'\rangle$ and $\rho_A \coloneqq \pi_A \circ \delta_A$ 
provides a geometric partial comodule structure on $A$ in the category of $\K$-algebras.

\begin{theorem}\label{prop:globPCA}
Let $H$ be a left flat $\K$-bialgebra and $(A,\delta_A)$ an algebraic partial $H$-comodule algebra. Then the associated geometric partial $H$-comodule algebra $(A,A\bul H,\pi_A,\rho_A)$ is globalizable. 
\end{theorem}

\begin{proof}
Let $Y$ be as in Corollary \ref{comodalgglob}.
In light of \eqref{eq:coass} and the definition \eqref{eq:piA} of $\pi_A$, the $\K$-linear map $\delta_A$ satisfies
\[\left(\rho_A\ot H\right)\circ\delta_A = \left(\pi_A\ot H\right) \circ \left(A \ot \Delta\right) \circ \delta_A.\]
As a consequence, there exists a unique $\K$-linear morphism $\vartheta:A \to Y$ such that $\kappa\circ \vartheta = \delta_A$ (that is, $\delta_A$ takes values in $Y$) and hence, since $\id_A = (A\ot  \varepsilon) \circ \delta_A$ in view of \ref{item:pca2}, 
$A\ot  \varepsilon$ is still surjective when restricted to $Y$. Moreover, as explained above the proposition, $\ker(\pi_A)$ is generated as an ideal in $A\ot H$ by the element $e'=1_A\ot 1_H-1_{A[0]}\ot 1_{A[1]}$. By denoting $1_A=1$, we find that
\begin{align*}
& (\delta\ot H-A\ot \Delta)(e') \\
 & \stackrel{\phantom{\eqref{eq:coass}}}{=} 1_{[0]}\ot 1_{[1]}\ot 1_H - 1\ot 1_H\ot 1_H - 1_{[0][0]}\ot 1_{[0][1]}\ot 1_{[1]} + 1_{[0]}\ot 1_{[1](1)}\ot 1_{[1](2)}\\
 & \stackrel{\eqref{eq:coass}}{=} - e'\ot 1_H - 1_{[0]}1_{[0]}\ot 1_{[1]}1_{[1](1)}\ot 1_{[1](2)} + 1_{[0]}\ot 1_{[1](1)}\ot 1_{[1](2)} \\
 & \stackrel{\phantom{\eqref{eq:coass}}}{=} (e'\ot 1_H)(1_{[0]}\ot 1_{[1](1)}\ot 1_{[1](2)} - 1 \ot 1_H \ot 1_H) \in \ker(\pi_A)\ot H
\end{align*}
Therefore, $e'\in Y$ and hence $A$ is globalizable by Corollary \ref{comodalgglob}.
\end{proof}


From now on (following \cite{AlvesBatista}), let us assume that $\K$ is a field.
The injective map $\vartheta:A\to Y$ from the proof of Theorem \ref{prop:globPCA} is clearly multiplicative (as it is induced by the multiplicative map $\delta_A$). One could consider the smallest subcomodule algebra 
of $Y$ containing $\vartheta(A)$. This leads to the notion of enveloping coaction in the sense of \cite{AlvesBatista}. More precisely, an enveloping coaction for an algebraic partial $H$-comodule algebra $A$ is a (global) comodule algebra $(B,\delta_B: b\mapsto b^{[0]}\ot b^{[1]})$ with an injective multiplicative map $\theta:A\to B$ such that
\begin{enumerate}[label=(\alph*),ref=(\alph*),leftmargin=0.7cm]
\item $\theta(A)$ is a unital right ideal of $B$ generated by $e \coloneqq \theta(1_A)$,
\item\label{item:gpca2} $B$ is generated by $\theta(A)$ as an $H$-comodule algebra and 
\item $(\theta \ot H) \circ \delta_A = (\theta(1_A)\ot H)\cdot \left(\delta_B \circ \theta\right)$ or, equivalently, for all $a\in A$
\begin{equation}\label{eq:thetadelta}
\theta(a_{[0]})\ot a_{[1]}= e\theta(a)^{[0]}\ot \theta(a)^{[1]}.
\end{equation}
\end{enumerate}
In \cite[Theorem 4]{AlvesBatista}, it was proven that the enveloping coaction of any algebraic partial comodule algebra exists. More precisely, $B$ can be constructed as the $H$-subcomodule algebra of $A\ot H$ (which is a right $H$-comodule via $A\ot \Delta$) generated by all elements of the form $a_{[0]}\ot a_{[1](1)}f(a_{[1](2)})$, with $a\in A$ and $f\in H^*$. The morphism $\theta:A\to B$ is then given by $\delta_A$ and $e=1_{A[0]}\ot 1_{A[1]}$. 
One may check that $\theta$ corestricts to an isomorphism $\theta:A\to eB$. Hence we can consider the projection of algebras $p:B\to A, b\mapsto \theta^{-1}(eb)$. In the realization of $B$ as above, we find that $p(\sum_i a_i\ot h_i)=1_{A[0]}a_i\varepsilon(1_{A[1]}h_i)$ for any $\sum_ia_i\ot h_i\in B$.


\begin{proposition}\label{pr:comodalg}
Given a partial $H$-comodule algebra $A$ over a field $\K$, the enveloping coaction $B$ of $A$ in the sense of \cite{AlvesBatista} is a subcomodule algebra of the globalization $Y_A$ of $A$. Namely, there is a unique comodule algebra monomorphism $j:B \to Y_A$ such that one of the following (equivalent) conditions hold:
\begin{enumerate}[label=(\Roman*),ref=(\Roman*),leftmargin=0.7cm]
\item $\epsilon_A\circ j=p$;
\item\label{item:comodalg2} $\kappa\circ j=(p\ot H)\circ \delta_B$.
\end{enumerate}
In particular, $B$ is co-generated by $A$ in the sense of \cite[Definition 2.10]{Saracco-Vercruysse} (i.e. $(p\ot H)\circ \delta_B$ is a monomorphism).
\end{proposition}

\begin{proof}
In view of the foregoing discussion, we can identify $A$ with $eB$, $\theta$ with the inclusion map and $p$ with left multiplication by $e$. Under this identification, we find
\begin{align*}
\big((\delta_A\ot H)\circ (p\ot H)\circ \delta_B\big)(a) & \stackrel{\phantom{\eqref{eq:coass}}}{=} (ea^{[0]})_{[0]}\ot (ea^{[0]})_{[1]}\ot a^{[1]} \stackrel{\eqref{eq:thetadelta}}{=} a_{[0][0]}\ot a_{[0][1]}\ot a_{[1]} \\
 & \stackrel{\eqref{eq:coass}}{=} 1_{A[0]}a_{[0]}\ot 1_{A[1]}a_{[1](1)}\ot a_{[1](2)} \qquad \text{and} \\
\big((A\ot \Delta)\circ (p\ot H)\circ \delta_B\big)(a) & \stackrel{\phantom{\eqref{eq:coass}}}{=} e a^{[0]}\ot a^{[1]}_{(1)}\ot a^{[1]}_{(2)} \stackrel{\eqref{eq:thetadelta}}{=} a_{[0]}\ot a_{[1](1)}\ot a_{[1](2)}
\end{align*}
for all $a\in A$. By using the fact that $\ker(\pi_A)$ is generated by the element $1_A\ot 1_H-1_{A[0]}\ot 1_{A[1]}$, we can conclude that $\varkappa \coloneqq (p\ot H)\circ \delta_B : B \to A \ot H$ satisfies
$$(\rho_A\ot H)\circ \varkappa \circ \theta = (\pi_A\ot H)\circ(\delta_A\ot H)\circ \varkappa \circ \theta = (\pi_A\ot H)\circ (A\ot \Delta)\circ \varkappa \circ \theta.$$
Since $(\rho_A\ot H)\circ \varkappa$ and $(\pi_A\ot H)\circ (A\ot \Delta)\circ \varkappa$ are $H$-comodule algebra maps and since $B$ is generated by $\theta(A)$ as $H$-comodule algebra, we can conclude that $(\rho_A\ot H)\circ \varkappa = (\pi_A\ot H)\circ (A\ot \Delta)\circ \varkappa$.
Thus, being $Y_A$ given by the equalizer \eqref{globcomalg}, we find that there is a unique morphism of $H$-comodule algebras $j:B\to Y_A$ such that $\kappa\circ j=\varkappa$. By composing the last identity with $A\ot \varepsilon$, we find that $\epsilon_A\circ j= p$.

Finally remark that by the identity \ref{item:comodalg2} and the fact that $\kappa$ is a monomorphism, $j$ is injective if and only if $\varkappa = (p\ot H)\ot \delta_B$ is so. Consider the explicit realization of $B$ as subcomodule algebra of $A\ot H$ from \cite[proof of Theorem 4]{AlvesBatista} as recalled above. For any element $a_{[0]}\ot a_{[1]}\in \theta(A)$, we find 
$$\varkappa(a_{[0]}\ot a_{[1]})=1_{A[0]}a_{[0]}\varepsilon(1_{A[1]}a_{[1](1)})\ot a_{[1](2)} \stackrel{\eqref{eq:coass}}{=} a_{[0][0]}\varepsilon(a_{[0][1]})\ot a_{[1]}=a_{[0]}\ot a_{[1]},$$
whence $\varkappa = (p\ot H)\ot \delta_B$ coincides with the inclusion $B\subset A\ot H$ on $\theta(A)$ and since both maps are comodule algebra morphisms and $B$ is generated by $\theta(A)$ as comodule algebra, we find that $(p\ot H)\ot \delta_B$ coincides with the inclusion on the whole of $B$. In particular, $(p\ot H)\ot \delta_B$ is injective.
\end{proof}

We conclude this paper by providing some examples that show how, in general, the enveloping coaction differs from the globalization. 


\begin{example}[{\cite[Example 1]{AlvesBatista}}]\label{ex:AB1}
Let $G$ be a finite group. If $N$ is a normal subgroup of $G$ and $\mathsf{char}(\K)\nmid |N|$, then $t=\frac{1}{|N|}\sum_{n\in N}n\in \K N$ is a central idempotent in $H \coloneqq \K G$. Notice also that $t$ is an integral in $\K N$, in the sense that $nt= t = tn$ for all $n\in N$.
Let $A\coloneqq t \, \K G$ be the (unital) ideal generated by $t$, let $p:H \to A, h\mapsto th,$ be the canonical projection and let $\iota:A \to H$ be the inclusion. Consider the partial $\K G$-coaction on $A$ given by 
\begin{equation}\label{eq:deltaAex1}
\delta_A(tg) = (t \ot 1)\Delta(tg) = tt_1g\ot t_2g = tg\ot tg. 
\end{equation}
In this case, $A \cong \K[G/N]$ (the group algebra over $G/N$) with partial coaction given by the composition $A \xrightarrow{\Delta_A} A \ot A \xrightarrow{A \ot \iota} A \ot H$ and $A \bul H$ given by
\[
\frac{A\ot H}{\langle t \ot 1 - \delta_A(t) \rangle} \stackrel{\eqref{eq:deltaAex1}}{=} \frac{A\ot H}{\langle t \ot 1 - t \ot t \rangle} = \frac{A \ot H}{(t\ot 1 - t \ot t)(A\ot H)} = \frac{A\ot H}{A\ot (1-t)H},
\]
which is isomorphic to $A\ot A$ via the factorization through the quotient of the projection $\pi_A \coloneqq (A\ot p) : A\ot H \to A\ot A$. 

To construct the globalization of $A$, choose a family $\{g_1,\ldots, g_r\}$ of representatives of the right cosets of $N$ in $G$ \big(\ie $G=\underset{i=1}{\overset{r}{\sqcup}} Ng_i$\big) and observe that $\{tg_1, \ldots, tg_r\}$ forms a basis of $A$.
Since $\rho_A = \pi_A \circ \delta_A$, we have that
$
z = \sum_{i=1}^r\sum_{g\in G} c_{i,g}(tg_i)\ot g \in A \ot H
$
belongs to $Y_A = \mathsf{Eq}(\rho_A\ot H,(\pi_A\ot H)\circ(A\ot \Delta))$ if and only if 
\[
\sum_{i=1}^r\sum_{g\in G} c_{i,g}\left(tg_i \ot tg_i \ot g - tg_i \ot g \ot g\right) \in \ker(\pi_A \ot H) = A \ot (1-t)H \ot H,
\]
if and only if $tg_i - g \in (1-t)H$, for all $g\in G$ and all $i=1,\ldots,r$ such that $c_{i,g} \neq 0$.
Thanks to the fact that $nt= t = tn$ for all $ n\in N$, one may now check directly that, in fact, $z\in Y_A$ if and only if $c_{i,g}\neq 0$ only for $g\in Ng_i$, that is to say, $Y_A = \mathsf{span}_{\K}\left\{ tg \ot g \mid g \in G\right\}$, which is the enveloping coaction as shown in \cite{AlvesBatista}.
\end{example}

\begin{example}[{\cite[Example 2]{AlvesBatista}}]\label{ex:AB2}
Let $H_4$ be Sweedler's four dimensional Hopf algebra, $H_4 = \K\langle g,x \mid g^2=1 , x^2 = 0, xg = -gx\rangle$, with $g$ group-like and $\Delta(x) = x\ot 1- g\ot x$, $\varepsilon(x)=0$. For any $\alpha\in\K$, the element $f = \frac{1}{2}\left(1 + g + \alpha gx\right)$ is an idempotent in $H_4$ and, by identifying $H_4$ with $\K\ot H_4$ in the canonical way, the assignment 
$
\delta_{\K} : \K \to H_4, \lambda \mapsto \lambda f,
$ 
defines a structure of partial $H_4$-comodule algebra on $\K$. In this case, $f = \delta_\K(1_\K)$ 
and $\K\bul H_4 = \K\ot H_4/\langle 1-\delta(1)\rangle = H_4/\langle 1-f \rangle$. A straightforward check reveals that $\langle 1-f \rangle = \ker(\varepsilon)$ 
and hence $\K\bul H_4 \cong \K$ via $\varepsilon: H_4 \to \K$. Therefore, $\K$ has the trivial partial $H_4$-comodule structure $(\K,H_4,\varepsilon,\id_\K)$ and so $Y_\K = H_4$, which strictly contains $\mathsf{span}_\K\{1,f\}$, that is the enveloping coaction according to \cite{AlvesBatista}.

In a similar way, one can check that the globalization of the partial comodule algebra from \cite[Example 3]{AlvesBatista} strictly contains the enveloping coaction.
\end{example}

\addtocontents{toc}{\protect\setcounter{tocdepth}{2}}

\end{document}